\documentclass[11pt]{amsart}

\usepackage[T1]{fontenc}
\usepackage{lmodern}
\usepackage{microtype}
\usepackage{mathtools,amssymb}
\usepackage{xcolor}
\usepackage[colorlinks=true,linkcolor=blue!55!black,
  citecolor=blue!55!black,urlcolor=blue!55!black]{hyperref}

\allowdisplaybreaks

\newtheorem{theorem}{Theorem}[section]
\newtheorem{proposition}[theorem]{Proposition}
\newtheorem{lemma}[theorem]{Lemma}
\newtheorem{corollary}[theorem]{Corollary}
\theoremstyle{remark}

\newcommand{\T}{\mathbb T}
\newcommand{\Sone}{\mathbb S^1}
\newcommand{\Z}{\mathbb Z}
\newcommand{\C}{\mathbb C}
\newcommand{\dd}{\,\mathrm d}
\newcommand{\e}{\varepsilon}
\DeclarePairedDelimiter{\abs}{\lvert}{\rvert}
\DeclarePairedDelimiter{\norm}{\lVert}{\rVert}
\numberwithin{equation}{section}
\title[Fourier summation formulas for degree]
{Nonexistence of universal Fourier \\ summation formulas for the degree \\ below the H\"older threshold $\alpha=1/3$}

\author{Micha{\l} Cieszy\'nski}
\address{Faculty of Mathematics, Informatics and Mechanics,
University of Warsaw, Banacha 2, 02-097 Warsaw, Poland}
\email{\href{mailto:m.cieszynski@student.uw.edu.pl}{m.cieszynski@student.uw.edu.pl}}

\begin{document}

\begin{abstract}
Let $(\sigma_{n,\e})_{n\in\Z,\,0<\e<1}$ be a summation process in the sense of Brezis, that is:
\begin{equation*}
  \forall \varepsilon\in(0,1), \ \sup_{n\in\Z}\abs{n\sigma_{n,\e}}<\infty
  \qquad  \forall n\in\Z \lim_{\varepsilon\rightarrow0^+}\sigma_{n,\e}=1.    \end{equation*}
We prove that for every $0<\alpha<1/3$ there exists $f\in C^{0,\alpha}(\Sone;\Sone)$ such that
\begin{equation*}
    \sum_{n\in\Z}\sigma_{n,\varepsilon}n|\hat f(n)|^2
\end{equation*}
does not converge to $\deg f$ as $\e\rightarrow 0^+$. This gives a negative answer to Open Problem 5.6 from Brezis's list of favourite open problems for all $p>3$ and $0<\alpha<1/3$, leaving the endpoint $C^{0,1/3}$ unresolved.
We also observe that, under their literal formulations, Open Problems 5.7
and 5.8 have immediate positive and negative answers, respectively. The
proof of the main result combines degree-zero quotients of Blaschke
factors with a Baire category argument.
\end{abstract}

\maketitle

\section{Introduction and main results}
For $f\in C(\Sone;\Sone)$, choose a continuous lift $\varphi:[0,2\pi]\rightarrow\mathbb{R}$ satisfying $f(e^{it})=e^{i\varphi(t)}$ and define
\[\deg f\coloneqq \frac{\varphi(2\pi)-\varphi(0)}{2\pi}.\]
For $f\in \operatorname{VMO}(\Sone;\Sone)$ we use the Brezis--Nirenberg $\operatorname{VMO}$ degree; see~\cite{BrezisNirenberg1995}.\\
Let $f:\Sone\to\Sone$ and write
\[
  \widehat f(n)=\int_{\Sone} f(z)\overline z^{\,n}\dd \mu(z),
  \qquad n\in\Z,
\]
where $\mu$ denotes the arclength measure on $\Sone$ normalised so that $\mu(\Sone)=1$. If
\mbox{$f\in H^{1/2}(\Sone;\Sone)$}, then the Fourier formula
\begin{equation}
  \deg f=\sum_{n\in\Z}n\abs{\widehat f(n)}^2                 \label{eq:degree-formula}
\end{equation}
is absolutely convergent.  Formula~\eqref{eq:degree-formula} was discovered by
Brezis following a question of Gelfand; see
Brezis--Nirenberg~\cite{BrezisNirenberg1995} and
Brezis--Mironescu~\cite[Chapter~12]{BrezisMironescu2021}.

Outside $H^{1/2}$ the series in~\eqref{eq:degree-formula} need not converge, but it is unknown in general whether the degree can be recovered from absolute values of Fourier coefficients alone.
The first partial answers were given by Korevaar~\cite{Korevaar1999}, he considered two summation processes, namely symmetric partial sums:
\[\lim_{N\rightarrow\infty}\sum_{n=-N}^N n|\hat f(n)|^2\]
and Abel summation:
\[\lim_{r\rightarrow 1^-}\sum_{-\infty}^\infty n|\hat f(n)|^2r^{|n|}\]
He proved that both formulas above equal $\deg f$ for $f\in C\cap BV$, but each formula fails for some $f\in C$. Brezis~\cite{Brezis2006} obtained the formula
\begin{equation}
    \deg f=\lim_{\varepsilon\rightarrow 0^+}\sum_{n\in\mathbb Z\setminus \{0\}} n |\hat f(n)|^2\frac{\sin(n\varepsilon)}{n\varepsilon}, \qquad \text{for } f\in W^{1/3,3}(\Sone;\Sone).
    \label{eq:brezisdegree}
\end{equation}
while Kahane~\cite{Kahane2005,Kahane2011} proved that there exists $f\in C^{0,1/3}(\Sone;\Sone)$ such that the formula in $\eqref{eq:brezisdegree}$ does not hold for $f$.
Bourgain and
Kozma~\cite{BourgainKozma2007} constructed maps $f,g\in C(\Sone;\Sone)$, such that \[|\hat f(n)|=|\hat g(n)| \ \forall_{n\in\mathbb Z}, \qquad \deg f\neq\deg g\]  Their result rules out recovery of degree from Fourier moduli
on all of $C(\Sone;\Sone)$, but not on $W^{1/p,p}$ for $p>3$.

Brezis formulated a question in terms of summation processes
\cite[Open Problem~5.6]{Brezis2023}.  He defines a summation process to be a family
$(\sigma_{n,\e})_{n\in\Z,\,0<\e<1}$ such that
\begin{align}
  \sup_{n\in\Z}\abs{n\sigma_{n,\e}}&<\infty
  &&\text{for every $\e\in(0,1)$},                         \label{eq:summation-1}\\
  \sigma_{n,\e}&\longrightarrow1
  &&\text{as $\e\rightarrow 0^+$, for every $n\in\Z$.}         \label{eq:summation-2}
\end{align}
The problem asks, for fixed $p>3 \, (\text{resp. } 0<\alpha\le 1/3)$, whether one can choose such a process,
depending only on $p$ (resp. $\alpha$), for which
\begin{equation}
  \sum_{n\in\Z}n\sigma_{n,\e}\abs{\widehat f(n)}^2
  \longrightarrow\deg f                                  \label{eq:recovery}
\end{equation}
for every $f\in W^{1/p,p}(\Sone;\Sone)$ (resp. $f\in C^{0,\alpha}(\Sone;\Sone)$).  We answer this question negatively for $p>3$ and $\alpha< 1/3$. \\
For a further generalization of Brezis's summation formula \eqref{eq:recovery}, for a sequence $b\in\ell^\infty(\Z;\C)$ we define
\[T_b(f)\coloneqq\sum_{n\in\Z}b_n|\hat f(n)|^2.\]

\begin{theorem}\label{thm:main}
Let $0<\alpha<1/3$. For every sequence
$(b^{(j)})_{j\in\mathbb N} 
\subset\ell^\infty(\mathbb Z;\mathbb C)$, one has
\begin{equation}
  \sup_{f\in C^{0,\alpha}(\Sone;\Sone)}\limsup_{j\to\infty}
  \abs{T_{b^{(j)}}(f)-\deg f}=\infty.                       \label{eq:main}
\end{equation}
In particular, no sequence $T_{b^{(j)}}$ converges pointwise to $\deg f$ on $C^{0,\alpha}(\Sone;\Sone)$.
\end{theorem}

\begin{corollary}\label{cor:brezis}
For $0<\alpha<1/3$, no summation process satisfying~\eqref{eq:summation-1}--\eqref{eq:summation-2} satisfies
\eqref{eq:recovery} for every $f\in C^{0,\alpha}(\Sone;\Sone)$.  In particular, Brezis's Open
Problem~5.6 has a negative answer in  $C^{0,\alpha}(\Sone;\Sone)$ for every $0<\alpha<1/3$ and in $W^{1/p,p}(\Sone;\Sone)$ for every $p>3$.
\end{corollary}

The proof of Theorem~\ref{thm:main} is by contradiction. A Baire category argument gives a neighbourhood of some $F$ on which the errors $T_{b^{(j)}}(f)-\deg f$ are eventually uniformly bounded. Proposition~\ref{prop:anchored}, using degree-zero quotients of Blaschke factors and rotation averaging, constructs arbitrarily small degree-preserving perturbations of $F$ for which $T_{b^{(j)}}$ changes by an arbitrarily large amount.

Open Problems~5.7 and~5.8 in~\cite{Brezis2023} reverse the order of quantifiers by allowing the summation process to depend on the individual map. Section~\ref{sec:individual-processes} gives an immediate affirmative answer to~5.7 and hence a negative answer to~5.8.

The restriction $\alpha<1/3$ enters through both the decay $(1-t^2)^{1-3\alpha}$ in~\eqref{eq:S-decay} and the divergence $n^{1-3\alpha}\to\infty$ in~\eqref{eq:growth}. Both become borderline at $\alpha=1/3$. Thus the argument gives no conclusion for $C^{0,1/3}(\Sone;\Sone)$ and is consistent with the positive formula on $W^{1/3,3}(\Sone;\Sone)$ proved in~\cite{Brezis2006}; see also~\cite[Theorem~5.6]{Brezis2023}.

Theorem~\ref{thm:main} concerns linear functionals of the squared Fourier moduli. It does not produce two maps with identical Fourier moduli and different degrees in the function spaces considered here. Thus it settles the summation problem, but not Open Problem~5.5 in~\cite{Brezis2023}, concerning whether the Fourier moduli determine the degree on $W^{1/p,p}$ for $p>3$ or on $C^{0,\alpha}$ for $0<\alpha\le1/3$.

\subsection*{Acknowledgements} The project is partially supported by National Science Centre grant 2022/01/1/ST1/00021 and by IMAI Centre (University of Warsaw). \\ The author thanks his advisor Katarzyna Mazowiecka for drawing his attention to this problem.

\subsection*{Use of generative AI} ChatGPT (OpenAI) was used to generate preliminary drafts of the proofs in this manuscript. The author subsequently verified each argument in detail, revised the proofs where necessary, checked the cited sources, and determined the final formulation of all results. The author takes sole responsibility for the correctness and content of the manuscript.

\section{Preliminaries}

We identify $\Sone$ with the unit circle. For $x,y\in\Sone$ we let $d(x,y)$ denote the geodesic distance. For $0<s,\alpha<1$ and $1<p<\infty$, set
\begin{align*}
  [f]_{C^{0,\alpha}}
  &:=\sup_{z\ne w}\frac{\abs{f(z)-f(w)}}{d(z,w)^\alpha},\\
  [f]_{W^{s,p}}^p
  &:=\int_{\Sone}\int_{\Sone}
  \frac{\abs{f(z)-f(w)}^p}{d(z,w)^{1+sp}}
  \dd\mu(z)\dd\mu(w).
\end{align*}
We equip $C^{0,\alpha}$ with $\norm f_\infty+[f]_{C^{0,\alpha}}$ and
$W^{s,p}$ with $\norm f_{L^p}+[f]_{W^{s,p}}$.

For $0<s<\alpha<1$, $1<p<\infty$ we have
\begin{equation}
  C^{0,\alpha}(\mathbb S^1;\mathbb C)
    \hookrightarrow W^{s,p}(\mathbb S^1;\mathbb C).  
    \label{eq:embedding}
\end{equation} 
Indeed,
\begin{align*}
[f]_{W^{s,p}}^p
&=\int_{\mathbb S^1}\int_{\mathbb S^1}
\frac{|f(z)-f(w)|^p}
     {d(z,w)^{1+sp}}
\dd\mu(z)\dd\mu(w) \\
&\leq
[f]_{C^{0,\alpha}}^p
\int_{\mathbb S^1}\int_{\mathbb S^1}
d(z,w)^{(\alpha-s)p-1}
\dd\mu(z)\dd\mu(w) \\
&\leq
C_{s,p,\alpha}[f]_{C^{0,\alpha}}^p,
\end{align*}
where the last integral is finite because $(\alpha-s)p>0$.

When $sp>1$, the maps
are continuous and $\deg f$ is the ordinary topological degree. When $sp=1$, $W^{s,p}(\Sone)$ embeds into $\operatorname{VMO}(\Sone)$, and
$\deg f$ denotes the Brezis--Nirenberg $\operatorname{VMO}$ degree which agrees with the ordinary degree on continuous maps. In the cases $sp\geq1$ considered below, the degree is locally
constant in the corresponding function-space topology and satisfies
\begin{equation}
  \deg(fg)=\deg f+\deg g,
  \quad \deg(\overline f)=-\deg f,
  \quad \deg(z^nf)=n+\deg f.                              \label{eq:degree-identities}
\end{equation}
For the critical Sobolev case these facts follow from the $\operatorname{VMO}$ degree theory
of~\cite{BrezisNirenberg1995}; see also
\cite[Chapters~5 and~12]{BrezisMironescu2021}.  The metric space $C^{0,\alpha}(\Sone;\Sone)$ is
complete because it is a closed subset of the ambient Banach space $C^{0,\alpha}(\Sone;\C)$.
For $b\in\ell^\infty(\Z;\C)$, Parseval's identity and the Cauchy--Schwarz
inequality give
\begin{equation}
\begin{split}
  \abs{T_b(f)-T_b(g)}
  &\leq \norm b_{\ell^\infty}
  \sum_n\abs{\abs{\widehat f(n)}^2-\abs{\widehat g(n)}^2}\\
  &\leq \norm b_{\ell^\infty}
  (\norm f_2+\norm g_2)\norm{f-g}_2\\
  &\leq \norm b_{\ell^\infty}
  (\norm f_2+\norm g_2)\norm{f-g}_\infty.
\end{split}                                                \label{eq:T-continuity}
\end{equation}
Thus every $T_b$ is continuous on $C^{0,\alpha}(\Sone;\Sone)$.

We shall also use continuity of multiplication. For \(F,G\in C^{0,\alpha}(\Sone;\Sone)\), we have
\[
\begin{aligned}
&F(z)(G(z)-1)-F(w)(G(w)-1)\\
&\qquad=F(z)(G(z)-G(w))
       +(F(z)-F(w))(G(w)-1).
\end{aligned}
\]
Since \(|F|=1\), division by $d(z,w)^\alpha$ and passage to
the supremum yield
\begin{equation}\label{eq:multiplication}
  [F(G-1)]_{C^{0,\alpha}}
  \leq [G-1]_{C^{0,\alpha}}
  +\|G-1\|_\infty[F]_{C^{0,\alpha}}.
\end{equation}

Finally, if \(u\) is Lipschitz on \(\Sone\), then
\begin{equation}\label{eq:holder-interpolation}
  [u]_{C^{0,\alpha}}
  \leq2^{1-\alpha}\|u\|_\infty^{1-\alpha}
  \|u'\|_\infty^\alpha,
\end{equation}
where $u'$ denotes the derivative with respect to the arclength parameter. Indeed,
\begin{equation}\label{eq:pointwise-interpolation}
\begin{split}
         |u(x)-u(y)|
  \leq&\min\{2\|u\|_\infty,\|u'\|_\infty d(x,y)\}\\
  \leq&(2\|u\|_\infty)^{1-\alpha}
       (\|u'\|_\infty d(x,y))^\alpha. 
\end{split}
\end{equation}

The last inequality follows from
\(\min\{A,B\}\leq A^{1-\alpha}B^\alpha\) for \(A,B\geq0\);
division by \(d(x,y)^\alpha\) proves \eqref{eq:holder-interpolation}.

\section{Proof of the main result} \label{main}

Fix $0<\alpha<1/3$. The next lemma is the quantitative part of the argument.

\begin{lemma}\label{lem:test-map}
Let $m=(m_n)_{n\in\Z}$ be a bounded complex-valued sequence satisfying
$m_{-n}=-m_n$, and set
\[
  A(m):=\sup_{n\geq1}\frac{\abs{m_n}}{n^{3\alpha}}.
\]
There are constants $c,C,\eta_0>0$, depending only on $\alpha$, such that, for
every $0<\eta<\eta_0$, there exists $H\in C^\infty(\Sone;\Sone)$ for which
\begin{align}
  \deg H&=0,                                               \label{eq:H-degree}\\
  \norm{H-1}_{C^{0,\alpha}}&\leq C\eta,             \label{eq:H-small}\\
  \abs{\sum_{n\in\Z}m_n\abs{\widehat H(n)}^2}
  &\geq c\eta^3A(m).                                      \label{eq:H-response}
\end{align}
\end{lemma}

\begin{proof}
If $A(m)=0$, take $H=1$.  Otherwise choose $K\in\mathbb N$ such that, with
$M:=\abs{m_K}$,
\begin{equation}
  M\geq\frac12A(m)K^{3\alpha}.                                 \label{eq:choose-K}
\end{equation}
Then
\begin{equation}
  \abs{m_{K\ell}}\leq2M\ell^{3\alpha},
  \qquad \ell\geq1.                                      \label{eq:subsequence-bound}
\end{equation}

For $0\leq a,b<1$, define
\[
  B_a(z):=\frac{z-a}{1-az},
  \qquad h_{a,b}(z):=\frac{B_a(z)}{B_b(z)},
  \qquad \rho:=\frac{a-b}{1-ab}.
\]
Both $B_a$ and $B_b$ have degree one on $\Sone$.  A Laurent expansion gives,
for $\ell\geq1$,
\begin{align}
  \widehat{h_{a,b}}(\ell)
  &=\rho(1-a^2)a^{\ell-1},                                \label{eq:positive-coefficients}  \\
  \widehat{h_{a,b}}(-\ell)
  &=-\rho(1-b^2)b^{\ell-1} \qquad \text{for } b>0.                               \label{eq:negative-coefficients} 
\end{align}
For b=0 we have $\widehat{h_{a,0}}(-1)=-\rho$ and $\widehat{h_{a,0}}(-\ell)=0$ for $\ell\ge2$.

Set $H(z)=h_{a,b}(z^K)$.  Then $H$ is smooth, unimodular, and
\[\deg H=K(\deg B_a-\deg B_b)=0.\] Since the nonzero Fourier coefficients of
$H$ have indices divisible by $K$, oddness of $m$ and
\eqref{eq:positive-coefficients}--\eqref{eq:negative-coefficients} yield
\begin{equation}
  \sum_{n\in\Z}m_n\abs{\widehat H(n)}^2
  =\rho^2\bigl(S(a)-S(b)\bigr),                            \label{eq:exact-response}
\end{equation}
where
\[
  S(t):=(1-t^2)^2\sum_{\ell\geq1}m_{K\ell}t^{2\ell-2},
  \qquad S(0)=m_K.
\]
Since $3\alpha<1$ and $\sum t^{2\ell-2}(1-t^2)=1$, Jensen's inequality
gives
\[
  (1-t^2)\sum_{\ell\geq1}\ell^{3\alpha} t^{2\ell-2}\le\big(\sum_{\ell\geq1}\ell t^{2\ell-2}(1-t^2)\big)^{3\alpha}=(1-t^2)^{-3\alpha}.
\]
Thus by~\eqref{eq:subsequence-bound}
\begin{equation}
      \abs{S(t)}
  \leq2M(1-t^2)^2\sum_{\ell\geq1}\ell^{3\alpha} t^{2\ell-2}\le2M(1-t^2)^{1-3\alpha},
  \qquad 0\leq t<1.  \label{eq:S-decay}
\end{equation}
Choose $t_*\in(0,1)$, depending only on $\alpha$, so that
$(1-t_*^2)^{1-3\alpha}\leq1/8$.  It follows that
\begin{equation}
  \abs{S(0)-S(t_*)}\geq\frac34M.                          \label{eq:S-drop}
\end{equation}

Let $L=\operatorname{arctanh}t_*$, $\eta_0=\min\{L,\frac 1 2\}$. Fix $0<\eta<\eta_0$ and let $\delta=\eta K^{-\alpha}$. We have $0<\delta\leq\eta\leq\min\{L,1/2\}$.  Choose
\[
  J=\left\lceil\frac L\delta\right\rceil,
  \qquad \Delta=\frac LJ,
  \qquad t_j=\tanh(j\Delta),\quad 0\leq j\leq J.
\]
Then $\delta/2\leq\Delta\leq\delta$ and
$\delta/4\leq\tanh\Delta\leq\delta$.  By telescoping
\eqref{eq:S-drop}, for some $j\in\{0,\ldots,J-1\}$,
\begin{equation}
  \abs{S(t_{j+1})-S(t_j)}
  \geq\frac{3M}{4J}\geq\frac{3M\delta}{8L}.               \label{eq:one-step}
\end{equation}
Take $a=t_{j+1}$ and $b=t_j$.  The subtraction formula for $\tanh$ gives
$\rho=\tanh\Delta$, and hence~\eqref{eq:exact-response} and
\eqref{eq:one-step} imply
\[
  \abs{\sum_nm_n\abs{\widehat H(n)}^2}
  \geq c_\alpha M\delta^3
  \geq c_\alpha \eta^3A(m),
\]
where the last inequality follows from~\eqref{eq:choose-K}.

It remains to estimate $H-1$. The identity $B_a=B_\rho\circ B_b$ gives
\[h_{a,b}=k_\rho\circ B_b,\qquad k_\rho(z):=\frac{B_\rho(z)}z,\qquad k_\rho(z)-1=\frac{\rho(z-z^{-1})}{1-\rho z}\quad(z\in\Sone).\] 
Since $\abs{\rho}\leq\delta\leq1/2$, we have
\[
  \abs{1-\rho z}\geq1-\abs{\rho}\geq\frac12
  \qquad (z\in\Sone).
\]
Thus,
\begin{equation}
  \norm{k_\rho-1}_\infty
  \leq4\abs{\rho}.
  \label{eq:k-rho-uniform}
\end{equation}

Writing $z=e^{it}$ and differentiating gives

\[
  \frac{\dd}{\dd t}k_\rho(e^{it})
  =
  \frac{
    i\rho(e^{it}+e^{-it})(1-\rho e^{it})
    +i\rho^2e^{it}(e^{it}-e^{-it})
  }{(1-\rho e^{it})^2}.
\]
Using $\abs{\rho}\leq1/2$ and
$\abs{1-\rho e^{it}}\geq1/2$, we obtain
\begin{equation}
  \norm{\frac{\dd}{\dd t}(k_\rho(e^{it})-1)}_\infty
  \leq C\abs{\rho},
  \label{eq:k-rho-derivative}
\end{equation}
For $z=e^{it}$, direct differentiation of the Blaschke factor gives
\[
  \abs{\frac{\dd}{\dd t}B_b(e^{it})}
  =
  \frac{1-b^2}{\abs{e^{it}-b}^2}
  \leq
  \frac{1+b}{1-b}.
\]
Since $0\leq b\leq t_*<1$, it follows that
\[
  \sup_{t\in\mathbb R}
  \abs{\frac{\dd}{\dd t}B_b(e^{it})}
  \leq
  \frac{1+t_*}{1-t_*}.
\]
Combining this estimate with
\eqref{eq:k-rho-uniform}--\eqref{eq:k-rho-derivative} and the chain rule,
we obtain
\begin{equation}
  \norm{h_{a,b}-1}_\infty
  +
  \norm{(h_{a,b}-1)'}_\infty
  \leq C_\alpha\abs{\rho}.
  \label{eq:h-rho-estimate}
\end{equation}

Since $H(z)=h_{a,b}(z^K)$, we consequently have
\[
  \norm{H-1}_\infty
  \leq C_\alpha\abs{\rho},
  \qquad
  \norm{(H-1)'}_\infty
  \leq C_\alpha K\abs{\rho}.
\]
Applying the interpolation estimate
\eqref{eq:holder-interpolation} to $H-1$ gives
\begin{align*}
  [H-1]_{C^{0,\alpha}}
  &\leq
  2^{1-\alpha}
  \norm{H-1}_\infty^{1-\alpha}
  \norm{(H-1)'}_\infty^\alpha\\
  &\leq
  C_\alpha K^\alpha\abs{\rho}.
\end{align*}
Since $K\geq1$, this also yields
\[
  \norm{H-1}_{C^{0,\alpha}}
  =
  \norm{H-1}_\infty+[H-1]_{C^{0,\alpha}}
  \leq
  C_\alpha K^\alpha\abs{\rho}.
\]
Finally, $\abs{\rho}=\tanh\Delta\leq\delta$ and
$\delta=\eta K^{-\alpha}$, so
\[
  \norm{H-1}_{C^{0,\alpha}}
  \leq
  C_\alpha K^\alpha\delta
  =
  C_\alpha\eta.
\]
This proves~\eqref{eq:H-small} and completes the proof.
\end{proof}

Rotation averaging allows us to use Lemma~\ref{lem:test-map} near an
arbitrary map.

\begin{proposition}\label{prop:anchored}
For every $F\in C^{0,\alpha} (\Sone;\Sone) $ there are constants $c_F,r_F>0$ such that, for every
$b\in\ell^\infty(\Z;\C)$ and every $0<r<r_F$,
\begin{equation}
\begin{split}
  &\sup_{\substack{G\in C^\infty(\Sone;\Sone),\ \deg G=0\\
                    \norm{FG-F}_{C^{0,\alpha}}<r}}
  \abs{T_b(FG)-T_b(F)}\\
  &\hspace{25mm}\geq c_Fr^3
  \sup_{n\geq1}\frac{
  \abs{T_b(z^nF)-T_b(z^{-n}F)}}{n^{3\alpha}}.
\end{split}                                                \label{eq:anchored}
\end{equation}
\end{proposition}

\begin{proof}
Define the bounded complex odd sequence
\[
  m_n:=\frac{T_b(z^nF)-T_b(z^{-n}F)}2,
  \qquad n\in\Z.
\]
Boundedness follows from $\abs{T_b(z^nF)}\leq\norm b_{\ell^\infty}$.
Apply Lemma~\ref{lem:test-map}, with a parameter $\eta>0$ to be fixed, to
obtain a smooth degree-zero map $H$ satisfying
\begin{equation}
  \norm{H-1}_{C^{0,\alpha}}\leq C\eta,
  \qquad
  \abs{\sum_nm_n\abs{\widehat H(n)}^2}
  \geq c\eta^3\sup_{n\geq1}\frac{\abs{m_n}}{n^{3\alpha}}.       \label{eq:lemma-applied}
\end{equation}

For $\theta\in[0,2\pi]$, set
\[
  H_\theta^+(z):=H(e^{i\theta}z),
  \qquad
  H_\theta^-(z):=\overline{H(e^{i\theta}z)}.
\]
We claim that
\begin{equation}
\begin{split}
  &\frac1{2\pi}\int_0^{2\pi}
  \bigl(T_b(FH_\theta^\pm)-T_b(F)\bigr)\dd\theta\\
  &\qquad =
  \sum_{n\in\Z}\abs{\widehat H(n)}^2
  \bigl(T_b(z^{\pm n}F)-T_b(F)\bigr).
\end{split}
\label{eq:rotation-average}
\end{equation}

Indeed, we have 
\[
  \widehat{FH_\theta^+}(k)
  =
  \sum_{n\in\Z}
  \widehat H(n)e^{in\theta}\widehat F(k-n).
\]
Orthogonality in the variable $\theta$ therefore gives
\[
  \frac1{2\pi}\int_0^{2\pi}
  \abs{\widehat{FH_\theta^+}(k)}^2\dd\theta
  =
  \sum_{n\in\Z}
  \abs{\widehat H(n)}^2
  \abs{\widehat F(k-n)}^2.
\]
Multiplying by $b_k$, summing over $k$, and reindexing, we obtain
\begin{align*}
  \frac1{2\pi}\int_0^{2\pi}T_b(FH_\theta^+)\dd\theta
  &=
  \sum_{k,n\in\Z}
  b_k\abs{\widehat H(n)}^2
  \abs{\widehat F(k-n)}^2\\
  &=
  \sum_{n\in\Z}
  \abs{\widehat H(n)}^2T_b(z^nF).
\end{align*}
The interchange of the integral and the sums is justified by
\begin{align*}
  &\sum_{k,n\in\Z}
  \abs{b_k}\abs{\widehat H(n)}^2
  \abs{\widehat F(k-n)}^2\\
  &\qquad\leq
  \norm b_{\ell^\infty}
  \sum_{n\in\Z}\abs{\widehat H(n)}^2
  \sum_{k\in\Z}\abs{\widehat F(k-n)}^2\\
  &\qquad=
  \norm b_{\ell^\infty}
  \norm H_2^2\norm F_2^2
  =
  \norm b_{\ell^\infty}.
\end{align*}

For the minus sign, observe that
\[
  H_\theta^-(z)
  =
  \sum_{n\in\Z}
  \overline{\widehat H(n)}e^{-in\theta}z^{-n},
\]
and hence
\[
  \widehat{FH_\theta^-}(k)
  =
  \sum_{n\in\Z}
  \overline{\widehat H(n)}e^{-in\theta}
  \widehat F(k+n).
\]
The same orthogonality argument yields
\[
  \frac1{2\pi}\int_0^{2\pi}T_b(FH_\theta^-)\dd\theta
  =
  \sum_{n\in\Z}
  \abs{\widehat H(n)}^2T_b(z^{-n}F).
\]
Finally, Parseval's identity gives
\[
  \sum_{n\in\Z}\abs{\widehat H(n)}^2
  =
  \norm H_2^2
  =
  1.
\]
Subtracting $T_b(F)$ proves~\eqref{eq:rotation-average}.

Set
\[
  A_\pm:=
  \frac1{2\pi}\int_0^{2\pi}
  \bigl(T_b(FH_\theta^\pm)-T_b(F)\bigr)\dd\theta.
\]
By~\eqref{eq:rotation-average} and the definition of $m_n$,
\[
  \frac{A_+-A_-}{2}
  =
  \sum_{n\in\Z}m_n\abs{\widehat H(n)}^2.
\]
It follows that
\[
  \max\{\abs{A_+},\abs{A_-}\}
  \geq
  \abs{\sum_{n\in\Z}
  m_n\abs{\widehat H(n)}^2}.
\]
On the other hand,
\[
  \abs{A_\pm}
  \leq
  \sup_{\theta\in[0,2\pi]}
  \abs{T_b(FH_\theta^\pm)-T_b(F)}.
\]
The function of $\theta$ inside the supremum is continuous. Therefore,
for some $\theta\in[0,2\pi]$ and some
\[
  G\in\{H_\theta^+,H_\theta^-\},
\]
we have
\begin{equation}
  \abs{T_b(FG)-T_b(F)}
  \geq
  c\eta^3
  \sup_{n\geq1}\frac{\abs{m_n}}{n^{3\alpha}}.
  \label{eq:anchored-response}
\end{equation}

Rotation and conjugation preserve the norms of $H-1$.  Thus the
multiplication estimate~\eqref{eq:multiplication} together with \eqref{eq:lemma-applied} gives
$\norm{FG-F}_{C^{0,\alpha}}\leq C_F\eta$.  Moreover, $\deg G=0$ and therefore
$\deg(FG)=\deg F$ by~\eqref{eq:degree-identities}.  Choose
$\eta=r/(2C_F)$ and $r_F$ small enough so that Lemma~\ref{lem:test-map} applies.
Since $\abs{m_n}$ is one half of the numerator in~\eqref{eq:anchored}, the
factor $1/2$ can be absorbed into $c_F$.  This proves the proposition.
\end{proof}

\subsubsection*{Proof of Theorem~\ref{thm:main}}

Write $T_j=T_{b^{(j)}}$.  Suppose, contrary to~\eqref{eq:main}, that there
is $R<\infty$ such that
\begin{equation}
  \limsup_{j\to\infty}\abs{T_j(f)-\deg f}\leq R
  \qquad\text{for every }f\in  C^{0,\alpha}(\Sone;\Sone).                           \label{eq:bounded-error}
\end{equation}
For $N\geq1$, let
\[
  E_N:=\left\{f\in C^{0,\alpha}(\Sone;\Sone):
  \abs{T_j(f)-\deg f}\leq R+1\text{ for every }j\geq N\right\}.
\]
Each $E_N$ is closed by~\eqref{eq:T-continuity} and the continuity of the
degree.  Equation~\eqref{eq:bounded-error} gives $C^{0,\alpha}(\Sone;\Sone)=\bigcup_NE_N$.  Since
$C^{0,\alpha}(\Sone;\Sone)$ is complete, the Baire category theorem yields $F\in C^{0,\alpha}(\Sone;\Sone)$, $r>0$, and
$N\geq1$ such that
\begin{equation}
  \{f\in C^{0,\alpha}(\Sone;\Sone):\norm{f-F}_{C^{0,\alpha}}<r\}\subset E_N.                   \label{eq:baire-ball}
\end{equation}

For $n\in\Z$, set
\[
  m_n^{(j)}:=\frac{T_j(z^nF)-T_j(z^{-n}F)}2.
\]
For every fixed $n\geq1$,~\eqref{eq:bounded-error} and
$\deg(z^{\pm n}F)=\deg F\pm n$ imply, for all sufficiently large $j$,
\[
  \abs{m_n^{(j)}-n}\leq R+1.
\]
Since $1-3\alpha>0$, it follows that
\begin{equation}
  \sup_{n\geq1}\frac{\abs{m_n^{(j)}}}{n^{3\alpha}}
  \longrightarrow\infty.                                 \label{eq:growth}
\end{equation}
Indeed, given $L>0$, first choose a fixed $n$ such that
$(n-R-1)/n^{3\alpha}>L$, and then take $j$ sufficiently large for $|m_n^{(j)}|\ge n-R-1$.

Apply Proposition~\ref{prop:anchored} to $T_j$ at the fixed map $F$, with
the fixed radius $r_*:=\frac12\min\{r,r_F\}$.  By~\eqref{eq:growth}, there
are smooth maps $G_j:\Sone\to\Sone$ with $\deg G_j=0$ and
$\norm{FG_j-F}_{C^{0,\alpha}}<r_*$ such that
\begin{equation}
  \abs{T_j(FG_j)-T_j(F)}\ge\frac 1 2 c_F r_*^3\sup_{n\geq1}\frac{\abs{m_n^{(j)}}}{n^{3\alpha}}
  \longrightarrow\infty.             \label{eq:local-divergence}
\end{equation}

Both $F$ and $FG_j$ belong to the ball in~\eqref{eq:baire-ball}, and
$\deg(FG_j)=\deg F$.  Hence, for every $j\geq N$,
\[
\begin{split}
  \abs{T_j(FG_j)-T_j(F)}
  &\leq\abs{T_j(FG_j)-\deg F}
       +\abs{T_j(F)-\deg F}\\
  &\leq2(R+1),
\end{split}
\]
contradicting~\eqref{eq:local-divergence}.  This proves
Theorem~\ref{thm:main}. \qed

\section{Consequences for Brezis's problems}

For a summation process $\sigma$ satisfying
\eqref{eq:summation-1}--\eqref{eq:summation-2}, set
\[
  Q_\e^\sigma(f):=
  \sum_{n\in\Z}n\sigma_{n,\e}\abs{\widehat f(n)}^2.
\]
The series is absolutely convergent for every $f\in L^2(\Sone)$ by
\eqref{eq:summation-1}.

\begin{proof}[Proof of Corollary~\ref{cor:brezis}]
If~\eqref{eq:recovery} held for every $f\in C^{0,\alpha}(\Sone;\Sone)$, choose any sequence
$\e_j\rightarrow 0^+$ and set $b_n^{(j)}=n\sigma_{n,\e_j}$.  Condition
\eqref{eq:summation-1} gives a sequence $(b^{(j)})_{j\in\mathbb N}\subset\ell^\infty(\mathbb Z;\mathbb C)$, while
Theorem~\ref{thm:main} excludes pointwise convergence of
$T_{b^{(j)}}=Q_{\e_j}^\sigma$ to the degree. The contradiction proves the corollary for $C^{0,\alpha}$. For $W^{1/p,p}$ choose $\alpha$ such that $1/p<\alpha<1/3$, then the result follows from the embedding \[  C^{0,\alpha}(\mathbb S^1;\mathbb C)
    \hookrightarrow W^{1/p,p}(\mathbb S^1;\mathbb C)\]
given by~\eqref{eq:embedding}.
\end{proof}

\subsection*{Map-dependent summation processes}\label{sec:individual-processes}

Open Problem~5.7 in~\cite{Brezis2023} asks whether, for each individual
$f\in C(\Sone;\Sone)$, or more generally $f\in\operatorname{VMO}(\Sone;\Sone)$,
one may choose a summation process depending on $f$ that recovers its degree.
Open Problem~5.8 asks whether there is one map for which every summation
process fails.  Under the literal definition
\eqref{eq:summation-1}--\eqref{eq:summation-2}, the proofs are elementary. 

Since the summation process provided by this argument explicitly depends on the value of $\deg f$, the construction below should be understood as an
observation about the breadth of the definition of summation process, rather than as the construction of a classical summability method.

\begin{proposition}\label{prop:map-dependent}
Let $f\in L^2(\Sone)$ have infinite Fourier support.  For every $L\in\mathbb R$
there is a summation process, depending on $f$ and $L$, such that
\[
  Q_\e^\sigma(f)=L \qquad\text{for every }0<\e<1.
\]
Consequently, Open Problem~5.7 has an affirmative answer and Open
Problem~5.8 has a negative answer, both for continuous and for $\operatorname{VMO}$ maps.
\end{proposition}

\begin{proof}
Let $a_n=\widehat f(n)$.  For every $N\geq1$, choose $k_N\in\Z$ such that
$\abs{k_N}>N$ and $a_{k_N}\ne0$, and set
\[
  A_N:=\sum_{\abs n\leq N}n\abs{a_n}^2.
\]
For $N=N(\e):=\lfloor\e^{-1}\rfloor$, define
\[
  \sigma_{n,\e}:=
  \begin{cases}
    1,&\abs n\leq N,\\[1mm]
    \displaystyle\frac{L-A_N}{k_N\abs{a_{k_N}}^2},&n=k_N,\\[3mm]
    0,&\text{otherwise}.
  \end{cases}
\]
For each $\varepsilon$ the sequence has finite support, so it satisfies
\eqref{eq:summation-1}.  For every fixed $n$, it equals $1$ once
$N(\e)\geq\abs n$, and hence \eqref{eq:summation-2} holds.  Direct
substitution gives $Q_\e^\sigma(f)=L$. In particular by letting $L=\deg f$, we can recover the degree.

It remains only to consider a circle-valued map $f$ with finite Fourier support.
Let $S_f=\{n:\hat{f}(n)\neq 0\}$ and $m=\min S_f$, $M=\max S_f$. Then $\widehat{|f|^2}(M-m)=\hat f(M)\overline{\hat f(m)}\neq0$. Since $|f|^2=1$ this gives $M=m$ and $f(z)=\hat f(m)z^m$ with $|\hat f(m)|=1$.
For a monomial, every summation process gives
$Q_\e^\sigma(f)=m\sigma_{m,\e}\to m=\deg f$.  Thus every continuous or $\operatorname{VMO}$ circle map admits a process
depending on it that recovers the degree.  This proves the two assertions
about Open Problems~5.7 and~5.8.
\end{proof}

\end{document}